\documentclass[10pt]{amsart}

\usepackage{amsmath,amssymb,color,mathtools,listings,amsfonts,amsthm,stmaryrd,rotating,stmaryrd,enumitem}

\newcommand{\tr}{\triangleright}
\newcommand{\btr}{\blacktriangleright}
\newcommand{\cur}{\curvearrowright}
\def\g{{\mathfrak g}}

\newtheorem{theorem}{Theorem}
\newtheorem{cor}{Corollary}
\newtheorem{lemma}{Lemma}
\newtheorem{prop}{Proposition}
\newtheorem{rmk}{Remark}
\newtheorem{example}{Example}

\theoremstyle{definition}
\newtheorem{mydef}{Definition} 
\theoremstyle{definition}

\title{The Magnus expansion and Post-Lie algebras}

\author{Charles Curry, Kurusch Ebrahimi-Fard, Brynjulf Owren}

\address{Department of Mathematical Sciences, 
		Norwegian University of Science and Technology (NTNU), 
		7491 Trondheim, Norway.}

\email{charles.curry@ntnu.no} \email{kurusch.ebrahimi-fard@ntnu.no} \email{brynjulf.owren@ntnu.no}

\begin{document}


\begin{abstract}
We relate the classical and post-Lie Magnus expansions. Intertwining algebraic and geometric arguments allows to placing the classical Magnus expansion in the context of Lie group integrators.
\end{abstract}


\maketitle





\vspace{1cm}


\section{Introduction}
\label{sect:intro}

Consider the classical non-autonomous linear initial value problem
\begin{equation}
\label{matrix:ode}
	\dot{Y}(t) = A(t) Y(t),\quad Y(0)=y_0,
\end{equation}
where $Y(t)\in GL(n,\mathbb{R})$, and $A(t)\in\mathfrak{gl}(n,\mathbb{R})$. It is easy to see that if $A(t)=A$ is constant, the solution of \eqref{matrix:ode} is given by a matrix exponential, $Y(t)=\exp(tA)y_0$. In the non-autonomous case the exponential $\exp(\int_0^t A(s)ds)y_0$ fails in general to solve \eqref{matrix:ode} due to the non-commutativity of $A(t)$ at different times \cite{Epstein_63}. However, it turns out that by adding particular Lie polynomials to the integral $\int_0^t A(s)ds$ in the exponential $\exp(\int_0^t A(s)ds)y_0$, one can approximate the solution $Y(t)$ of \eqref{matrix:ode} in the non-autonomous case to an arbitrary high precision. Magnus \cite{Magnus1954} understood how to obtain the exact solution of \eqref{matrix:ode} by describing the Lie series yielding an exponential solution of \eqref{matrix:ode}. More precisely, he showed that the logarithm of $Y(t)$ can be characterised as the solution of a particular differential equation summarised in the next theorem.

\begin{theorem}[Classical Magnus expansion, \cite{Magnus1954}] \label{cme}
The solution of \eqref{matrix:ode} can be expressed in terms of a matrix exponential,
\[
	Y(t) = \exp\big(\Omega(A)(t)\big)y_0,
\]
where $\Omega(A)(t)$ is a matrix-valued function solving the differential equation
\begin{align}
	\dot{\Omega}(A)(t) = \mathrm{dexp}_{\Omega(A)}^{-1}(A)(t) 
	&:= \frac{\mathrm{ad}_{\Omega(A)}}{\exp(\mathrm{ad}_{\Omega(A)})-1}(A)(t) \label{firstP}\\
	&= \sum_{n\geq 0} \frac{B_n}{n!}\mathrm{ad}_{{\Omega}(A)}^{(n)}(A)(t)\label{cme:omega}
\end{align}
with initial value $\Omega(0)=0$, and with $B_n$ denoting the $n$th Bernoulli number.
\end{theorem}

Recall that $\mathrm{ad}_{U}^{(n)}(V):=[U,\mathrm{ad}_{U}^{(n-1)}(V)]$, and $\mathrm{ad}_{U}^{(0)}(V):=V$. We remark that $\mathrm{dexp}_{\Omega(A)}$ denotes the derivative of the exponential mapping at the point $\Omega(A)$. This is an invertible linear mapping for which the second equality in \eqref{firstP} above is classical.

Introducing a formal parameter $\lambda$, we see that from integrating and iteratively expanding \eqref{cme:omega} an infinite powers series $\Omega(\lambda A)(t)=\sum_{j>0}\Omega_j(A)(t)\lambda^j$ of nested integrals over iterated matrix commutators of the function $A(t)$ results. The coefficients of the first three orders of $\Omega(\lambda A)(t)$ are
\begin{align*}
	\Omega_1(A)(t) &= \int_0^t  dt_1 A(t_1)\\  
	\Omega_2(A)(t) &= - \frac{1}{2}\int_0^t dt_1 \int_0^{t_1} dt_2 [A(t_2),A(t_1)]\\
	\Omega_3(A)(t) &=  \frac{1}{4} \int_0^t dt_1 \int_0^{t_1} dt_2 \int_0^{t_2} dt_3 [[A(t_3),A(t_2)],A(t_1)]\\
				 &\qquad     +\frac{1}{12} \int_0^t dt_1 \int_0^{t_1} dt_2 \int_0^{t_1} dt_3 [A(t_2),[A(t_3),A(t_1)]].
\end{align*}
Over the decades, the Magnus expansion $\Omega(A)(t)$ has been studied in great detail from the point of view of applied mathematics \cite{BCOR2009,Iserles02,MielPleb70,MuOw99,Wilcox67}. More recently, deeper mathematical fine structures of the Magnus expansion have been unfolded using more advanced algebraic tools. See, e.g., the references \cite{BandSch,Chapoton09,CP13,EFMMK17,EM2009}. 

These mathematical structures of the Magnus expansion arise through an observation that links the above series, $\Omega(\lambda A)(t)$, to the general solution of certain fixed point equations known under the name pre-Lie Magnus expansion. Indeed, for two matrix-valued functions $U$ and $V$, we define the product \cite{AG1981,EM2009} 
\begin{equation}
\label{pre-Lie-basic}
	(U \cur V)(t) := \int_0^t ds[U(s),V(t)] .
\end{equation}
One verifies that the following identity 
\begin{equation}
\label{plid}
	(U \cur V) \cur W - U \cur (V \cur W) = (V \cur U) \cur W - V \cur (U \cur W)
\end{equation}
is satisfied, which is known as left pre-Lie relation. The product \eqref{pre-Lie-basic} therefore defines a pre-Lie algebra \cite{Cartier2011,Manchon2011} on the space of matrix-valued functions. Note that \eqref{plid} implies that the commutator $\llbracket U,V \rrbracket:=U \cur V - V \cur U$ yields a Lie bracket, i.e., pre-Lie algebras are Lie admissible. Let $\ell^{(n)}_{U\cur}(V) :=U \cur (\ell^{(n-1)}_{U\cur}(V))$, where {$\ell^{(0)}_{U\cur}(V):=V$}. Using the product \eqref{pre-Lie-basic} in expansion \eqref{cme:omega} yields the simple presentation of the Magnus expansion in terms of its underlying pre-Lie product \eqref{pre-Lie-basic}.

\begin{cor}[Pre-Lie Magnus expansion, \cite{AG1981,EM2009}] \label{cor:preLieMag}
The Magnus expansion solves the fixed point equation
\begin{equation}
\label{plm}
	\dot{\Omega}(A)(t) = \sum_{n \ge 0} \frac{B_n}{n!} \ell^{(n)}_{\dot{\Omega}(A)\cur}(A)(t).
\end{equation}
\end{cor}

In \cite{EM2009} it was observed that the pre-Lie relation \eqref{plid} allows to reduce the number of terms in the expansion $\Omega(\lambda A)(t)$. For instance, at order four, one finds that
$$
	\Omega_4(A) = \frac{1}{6}  ((A \cur A) \cur A ) \cur A + \frac{1}{12}  A \cur ((A \cur A ) \cur A).
$$
For details and more recent results we refer the reader to \cite{AlK_14,AlKManPat_17}.

\medskip

In this article, we show that by an appropriate `autonomization' the linear initial value problem \eqref{matrix:ode} can be presented as a Lie group integration problem along the lines of \cite{Iserles00}. In this context,  post-Lie algebras \cite{Voss1,LMK13} occur naturally replacing pre-Lie algebras. The classical Magnus expansion \eqref{cme:omega} is seen to be a special case of the so-called post-Lie Magnus expansion introduced in \cite{Voss1,EFLMK15,EFLMMK15}. This observation motivates to compare explicitly certain Runge--Kutta--Munthe-Kaas schemes with Magnus integrator schemes. It is hoped that this comment may help to relate the theories of Lie group integrators \cite{CHO2014,Iserles00} and Magnus integrators \cite{BCOR2009}, which have hitherto developed on separate lines. 

\medskip

The paper is organised as follows. In Section \ref{sect:prepost} we recall the notions of pre- and post-Lie algebras and show how the pre- and post-Lie Magnus expansions are naturally defined. Section \ref{sect:LGIMag} shows how Lie group integration and the post-Lie Magnus expansion can be related to the classical Magnus expansion. Section \ref{sect:RKMKMag} compares Runge--Kutta versus Magnus integration methods, and ends with a remark regarding continuous stage methods and a link to the classical Magnus expansion.

\medskip

\noindent {\bf{Acknowledgements}}: This research was supported by The Research Council of Norway, the FRIPRO programme, grant No.~231632 and by the European Union's Horizon 2020 research and innovation programme under the Marie Sklodowska--Curie grant agreement No.~691070. Finally, we thank the referees for pertinent suggestions and helpful remarks that resulted in an improved presentation of our work.


\section{Post- and Pre-Lie Magnus expansions}
\label{sect:prepost}

We briefly recall the algebraic notions of post- and pre-Lie algebras \cite{Cartier2011,Voss1,LMK13,Manchon2011}. All algebraic structures are considered over a base field $k$ of characteristic zero. For a $k$-vector space $A$ with general binary product $x \bullet y:=m_A(x \otimes y)$, $x,y \in A$, we define the associator  
$$
	{\rm{a}}_{\bullet}(x,y,z):=x \bullet (y \bullet z) - (x \bullet y) \bullet z.
$$
If $(A,\bullet)$ is an associative algebra, then ${\rm{a}}_{\bullet}(x,y,z)=0$. In our work, non-associative algebras will play a central role and the next definition provides the first example.  

\begin{mydef} \label{def:postLie} 
A \emph{post-Lie algebra} $(\mathfrak{g}, [\cdot,\cdot], \triangleright)$ consists of a Lie algebra $(\mathfrak g, [\cdot,\cdot])$ and a product $\triangleright  : {\mathfrak g} \otimes {\mathfrak g} \rightarrow \mathfrak g$ such that the following relations hold for all elements $x,y,z \in \mathfrak g$
\begin{align}
\label{postLie1}
	x \triangleright [y,z] &= [x\triangleright y , z] + [y , x \triangleright z],\\
\label{postLie2}
	[x,y] \triangleright z &= {\rm{a}}_{\triangleright  }(x,y,z) - {\rm{a}}_{\triangleright}(y,x,z).
\end{align}  
\end{mydef}

\begin{prop} \cite{LMK13} \label{prop:Liebracket}
Let $(\mathfrak g, [\cdot, \cdot], \triangleright )$ be a post-Lie algebra. For $x, y \in \mathfrak g$ the bracket
\begin{equation}
\label{postLie3}
	\llbracket x,y \rrbracket := x \tr y - y \tr x + [x,y]
\end{equation}
satisfies the Jacobi identity. The resulting Lie algebra is denoted $(\overline{\mathfrak g},\llbracket \cdot,\cdot \rrbracket)$. 
\end{prop}

We remark that the above notion of post-Lie algebra has an adjoint $(\mathfrak g,-[\cdot,\cdot],\btr)$. 

\begin{prop} \cite{LMK13} \label{prop:adjoint}
If $(\mathfrak{g}, [\cdot,\cdot], \triangleright)$ is a post-Lie algebra, then $(\mathfrak g, -[\cdot,\cdot],\btr)$, where
\[
	x \btr y := x \triangleright y + [x,y],
\]
is also a post-Lie algebra. 
\end{prop}

\begin{proof}
We repeat the proof as it is instructive to get acquainted with this set of identities. Indeed, we have 
\allowdisplaybreaks
\begin{eqnarray*}
	- x \btr [y,z] &=& - x \triangleright [y,z] - [x,[y,z]]\\
				&=&- [x\triangleright y,z] - [y,x\triangleright z] - [[x,y],z] - [y,[x,z]]\\
				&=&-[x \btr y,z]-[y,x \btr z].
\end{eqnarray*}
Next we calculate
\allowdisplaybreaks
$$
	(y \btr x) \btr z
	=(y \triangleright x) \triangleright z + [y,x] \triangleright z + [y \triangleright x,z] + [[y,x],z]
$$
and 
\begin{eqnarray*}
	y \btr (x \btr z)
	&=&y \triangleright (x\triangleright z) +  [y \triangleright x,z] +  [x,y \triangleright z]  + [y, x\triangleright z] + [y,[x,z]].
\end{eqnarray*}
The difference of the two terms yields the associator
$$
	{\rm{a}}_{\btr}(y,x,z) = {\rm{a}}_{\triangleright }(y,x,z) +  [x,y \triangleright z]  + [y, x\triangleright z] 
	+ [y,[x,z]] - [y,x] \triangleright z - [[y,x],z].
$$
Together with ${\rm{a}}_{\btr}(x,y,z)$ this gives 
\[
	{\rm{a}}_{\btr}(x,y,z) - {\rm{a}}_{\btr}(y,x,z) = -[x,y] \triangleright z - [[x,y],z]=-[x,y] \btr z.
\]
\end{proof}

\begin{example}\label{exm:connection}{\rm{
Let $\mathcal{X}(M)$ be the space of vector fields on a manifold $M$, equipped with a linear connection. For $X,Y \in \mathcal{X}(M)$, the covariant derivative of $Y$ in the direction of $X$, denoted $\nabla_X Y$, defines an $\mathbb{R}$-linear, non-associative binary product $X\triangleright Y:=\nabla_X Y$ on $\mathcal{X}(M)$. The torsion $T$ is a skew-symmetric tensor field of type $(1,2)$ {defined by}
\begin{align}
\label{torsion}
	T(X,Y) := X\triangleright Y - Y\triangleright X - \llbracket X,Y \rrbracket,
\end{align}
where the bracket $\llbracket \cdot,\cdot \rrbracket$ on the right is the Jacobi bracket of vector fields. The torsion admits a covariant differential $\nabla T$, a tensor field of type $(1,3)$. Recall the definition of the curvature tensor $R$, a tensor field of type $(1,3)$ given by
\[
	R(X,Y)Z := X\triangleright(Y\triangleright Z) - Y\triangleright(X\triangleright Z) 
	- \llbracket X,Y \rrbracket \triangleright Z.
\]
In the case of a flat connection with constant torsion, i.e., for $R=0=\nabla T$, we have that $(\mathcal{X}(M),-T(\cdot,\cdot),\triangleright)$ defines a post-Lie algebra. Indeed, the first Bianchi identity shows that $-T(\cdot,\cdot)$ obeys the Jacobi identity; as $T$ is skew-symmetric it therefore defines a Lie bracket. Moreover, flatness is equivalent to $(\ref{postLie2})$ as can be seen by inserting \eqref{torsion} into the statement $R=0$, whilst $(\ref{postLie1})$ follows from the definition of the covariant differential of $T$:
\[
	0 = \nabla T (Y,Z;X) = X \triangleright T(Y,Z) 
						- T(Y,X \triangleright Z) 
							- T(X \triangleright Y,Z).
\]
The formalism of post-Lie algebras assists greatly in understanding the interplay between covariant derivatives and integral curves of vector fields, which is central to the study of numerical analysis on manifolds.}}
\end{example}

Recall that for a Lie algebra $(\g,[\cdot,\cdot])$, its enveloping algebra is an associative algebra $(U(\g),\cdot)$ such that $\g \subset U(\g)$ and $[a,b] = a \cdot b - b \cdot a$ in $U(\g)$ \cite{Meinrenken2013}. As a Lie algebra $\g$ with product $\tr$, the enveloping algebra of a post-Lie algebra $(\mathfrak{\g}, [\cdot,\cdot], \tr)$ is $U(\g)$ together with an extension of the post-Lie product $\tr$ onto $U(\g)$ defined such that for all elements $u \in \g$ and $A,B\in U(\g)$
\begin{align*}
	u \tr (A \cdot B) &= (u \tr A)\cdot B + A \cdot (u \tr B)\\
	(u \cdot A)\tr B &= a_\tr(u,A,B).
\end{align*}
Proposition~\ref{prop:Liebracket} implies that a post-Lie algebra $(\mathfrak{\g}, [\cdot,\cdot], \tr)$ has a second Lie algebra structure $\bar{\g}$ associated to the bracket $\llbracket\cdot,\cdot\rrbracket$ defined in \eqref{postLie3}. As a vector space, its enveloping algebra $U(\bar{\g})$ is isomorphic to $U(\g)$. Moreover, the post-Lie product of $(\mathfrak{\g}, [\cdot,\cdot], \tr)$ lifted to $U(\g)$ permits to define another associative product on $U(\g)$, such that $U(\g)$ with this new product is isomorphic as a Hopf algebra to $U(\bar{\g})$. See \cite{EFLMK15,EFM2018} for details. By some abuse of notation we denote the lifted post-Lie product and new associative product on $U(\g)$ by $\tr$ respectively $\ast$. One of the central results is the following 

\begin{prop}\label{prop:product}\cite{EFLMK15} 
For $A,B,C \in U(\g)$ we have 
\begin{equation}
\label{product}
	(A * B) \tr C = A \tr (B \tr C). 
\end{equation}
\end{prop}

\medskip

From the point of view of Lie group integration we are mainly interested in comparing the two exponentials, $\exp^\centerdot$ and $\exp^*$, both defined in the appropriate completion of $U(\g)$. See \cite{EFM2018} for details. 

\begin{prop}\label{two exp}\cite{EFLMK15}
Consider the post-Lie algebra $(\mathcal{X}(G),-T(\cdot,\cdot),\triangleright)$ associated to a flat, constant torsion connection $\triangleright$, as per Example~\ref{exm:connection}. Then
\begin{enumerate}
\item The exponential $\exp^*(tf)$ is associated with the pullback of functions along integral curves, i.e., let
\[
	\dot{y}(t) = f(y(t)),\quad y(0)=y_0,
\]
then $\exp^*(tf):g\big(y_0\big)\mapsto g\big(y(t)\big)$ for any $g\in C^{\infty}(M)$.
\smallskip
\item The exponential $\exp^\centerdot(tf)$ gives the pullback of functions along geodesics of the product $\triangleright$, i.e., let $z(t)$ solve the geodesic equation 
\[
	\dot{z}(t) \triangleright \dot{z}(t) = 0,\quad z(0)=z_0,\quad \dot{z}(0) = f(z_0),
\]
then $\exp^\centerdot(tf):g\big(z_0\big)\mapsto g\big(z(t)\big)$ for any $g\in C^{\infty}(M)$.
\end{enumerate}
\end{prop}

The two exponentials are related by the post-Lie Magnus expansion:

\begin{lemma}[Post-Lie Magnus expansion, \cite{EFLMK15,EFMMK17}] \label{lem:postLieMag}
The exponentials $\exp^\centerdot$ and $\exp^*$ are related by the mapping $\theta(t):\g \to \g$,
\begin{equation}
\label{exprelation1}
	\exp^*(tf) = \exp^\centerdot(\theta(f)(t)),
\end{equation}
where $\theta(f)(t)$ solves the differential equation
\begin{equation}
\label{plm:eqn}
	\dot\theta(f)(t) 
	= \mathrm{dexp}^{\centerdot -1}_{\theta(f)} \big(\exp^\centerdot\big(\theta(f)\big)\tr f \big)(t), 
	\quad \theta(f)(0) = 0.
\end{equation}
For sufficiently small $t>0$, the map $\theta(f)(t)$ is invertible, and its inverse $\chi(f)(t)$ solves
\begin{equation}\label{plminv:eqn}
	\dot\chi(f)(t) = \mathrm{dexp}^{*-1}_{-\chi(f)} \big(\exp^*\big(-\chi(f)\big)\tr f \big)(t),
	\quad \chi(f)(0) = 0,
\end{equation} 
such that 
\begin{equation}
\label{exprelation1a}
	\exp^*(\chi(f)(t)) = \exp^\centerdot(tf).
\end{equation}
\end{lemma}

\begin{rmk} {\rm{
We remark that identity \eqref{exprelation1a} encodes backward error analysis for the forward exponential Euler method \cite{Voss1,LMK13}. Indeed, the Lie--Euler integration scheme is the numerical method that approximates solutions of the initial value problem \eqref{eq:mkform} below by following Lie group exponentials, i.e.,
\[
	y_1 = \exp\big(h\lambda_*f(y_0)\big)y_0.
\]
Backward error analysis studies properties of the approximate flow by finding modified vector fields $\tilde{f}$ for which the exact flow $\tilde{y}(t)$ coincides with $y_n$ at the discretisation times $t_n$. This problem is exactly that of finding $\chi(f)(t)$ for which $\exp^\centerdot(tf) = \exp^*(\chi(f)(t))$, for the post-Lie algebra structure on $C^{\infty}(M,\g)$ defined in \cite[Prop.~2.10]{LMK13}.

We also remark that the inverse of the post-Lie Magnus expansion has appeared in the study of isospectral flows \cite{EFLMMK15}.
}}
\end{rmk}

We introduce now the notion of Pre-Lie Magnus expansion. Observe that if the Lie algebra $\mathfrak g$ in Definition \ref{def:postLie} is trivial, i.e., if $[\cdot,\cdot]=0$, then the post-Lie algebra $(\mathfrak g, [\cdot, \cdot], \triangleright )$ reduces to $(\mathfrak g, \triangleright )$, which satisfies the definition of a pre-Lie algebra. The latter being defined as:

\begin{mydef} \label{def:preLie} 
A left \emph{pre-Lie algebra} $(\mathfrak{p}, \cur )$ consists of a vector space $\mathfrak p$ and a binary product $\cur  \colon {\mathfrak p} \otimes {\mathfrak p} \rightarrow \mathfrak p$ such that, for all elements $x,y,z \in \mathfrak p$ 
\begin{align}
\label{preLie}
	 {\rm{a}}_{\cur }(x,y,z) = {\rm{a}}_{\cur}(y,x,z).
\end{align}
\end{mydef}

\noindent Note that \eqref{preLie} in explicit form already appeared in \eqref{plid} above. The notion of right pre-Lie algebra is analogously defined with $ {\rm{a}}_{\curvearrowleft  }(x,y,z) = {\rm{a}}_{\curvearrowleft}(x,z,y)$ replacing \eqref{preLie}. 

\begin{example}\label{ex:preLie}{\rm{The natural geometric example of a pre-Lie algebra is given in terms of a differentiable manifold $M$ with a flat and torsion-free connection. Vanishing torsion and curvature are expressed in terms of the corresponding covariant derivation $\nabla$ on the space $\chi(M)$ of vector fields on $M$ satisfying the two equalities $\nabla_fg-\nabla_gf=[f,g]$ and $\nabla_{[f,g]}=[\nabla_f,\nabla_g]$, respectively. From this we deduce that $f \triangleright g:=\nabla_fg$ defines a left pre-lie algebra structure on $\chi(M)$.  Let $M=\mathbb{R}^n$ with its standard flat connection. For two vector fields $f(x)=\sum_{i=1}^nf^i(x)\frac{\partial}{\partial x^i}$ and $g(x)=\sum_{i=1}^ng^i(x)\frac{\partial}{\partial x^i}$ it follows that
\begin{equation}
\label{preLieVect1}
	(f \cur g) (x)=\sum_{i=1}^n\Big(\sum_{j=1}^n 
	f^j(x)\frac{\partial}{\partial x^j}g^i(x)\Big)\frac{\partial}{\partial x^i}.
\end{equation}
}}\end{example}

Let us look at Lemma \ref{lem:postLieMag} from the pre-Lie algebra point of view. Guin and Oudom \cite{OG2008} lifted the pre-Lie product of a pre-Lie algebra $(\mathfrak{p},  \cur)$ to the symmetric algebra $S(\mathfrak{p})$, and showed that this permits to introduce a noncommutative associative product on $S(\mathfrak{p})$, widely referred to as Grossman--Larson product, such that the resulting (Hopf) algebra $(S(\mathfrak{p}),\ast,\Delta)$ is isomorphic (as a Hopf algebra) to the enveloping algebra $U(\bar{\g})$. Note that the coproduct $\Delta$ is the usual unshuffle comultiplication defined on $S(\mathfrak{p})$. Recall that $\overline{\mathfrak g}$ is the Lie algebra with commutator bracket defined in terms of the pre-Lie product, see Definition \ref{def:preLie}. Chapoton and Patras \cite{CP13} proved an analog of Lemma \ref{lem:postLieMag}  in the pre-Lie case, i.e., the exponential $\exp^\centerdot$ and the Grossman--Larson exponential $\exp^*$, defined in an appropriate completion of $S(\mathfrak{p})$, are related by the pre-Lie Magnus expansion:
\begin{equation}
\label{exprelation2}
	\exp^*(\Omega_{\cur }(tx)) = \exp^\centerdot(tx),
\end{equation}
where $x \in {\g}$ and $t$ is a formal parameter. The pre-Lie Magnus expansion is defined in terms of the pre-Lie product $\cur$ 
\begin{align}
	\Omega_\cur (x) &:= \frac{\ell_{\Omega(x)\cur}}{\exp(\ell_{\Omega(x)\cur})-1}(x) 
	= \sum_{n \ge 0} \frac{B_n}{n!} \ell^{(n)}_{\Omega(x)\cur }(x).\label{pre-Lie-Magnus}
\end{align}
Recall that the compositional inverse $W(x):=\Omega^{-1}_\cur (x)$ is given by
\begin{equation}
\label{inverse-pre-Lie-Magnus}
	W(x) := \sum_{n \ge 0} \frac{1}{(n+1)!}  \ell^{(n)}_{x\cur }(x),
\end{equation}
such that
\begin{equation}
\label{exprelation3}
	\exp^*(tx) = \exp^\centerdot(W(tx)).
\end{equation} 
Returning to Example \ref{ex:preLie}, where the particular pre-Lie algebra \eqref{preLieVect1} is defined, we consider the initial value problem $\dot{y}=f(y)$, $y(0)=y_0$, where $f$ is a vector field on $\mathbb{R}^n$. The exact solution is given by the Grossman--Larson exponential, $y(t)=\exp^*(tf)$. Identity \eqref{exprelation2} defines the backward error analysis of the explicit forward Euler method on $\mathbb{R}^n$, i.e., $\exp^*(\Omega_{\cur }(hf))(y) = y + f(hy)$, where $h$ denotes the step size \cite{Chapoton02,CHV10,Murua06}.  On the other hand, from \eqref{exprelation3} we deduce that the Taylor expansion of the exact solution can be expressed as a series in iterated pre-Lie products 
$$
	\exp^*(tf)(y) 
	= y + W(tf)(y) 
	= y + tf(y) 
		+ \frac{t^2}{2} (f \cur  f)(y) 
			+  \frac{t^3}{6} (f \cur (f \cur f))(y) + \cdots.
$$


\section{Lie group integration and the Post-Lie Magnus expansion}
\label{sect:LGIMag}

Let $\mathcal{X}(M)$ be the space of vector fields on a manifold $M$. Lie group integration \cite{CHO2014,Iserles00} concerns initial value problems of the form 
\begin{equation}
\label{eq:mkform}
	\dot{y} = \big( \lambda_* f(y,t)\big)(y),\quad y(0)=p,
\end{equation}
where $f \colon M\times\mathbb{R_+} \rightarrow \mathfrak{g}$, and $\lambda_* \colon \mathfrak{g} \rightarrow \mathcal{X}(M)$ is the infinitesimal action of the Lie algebra $\mathfrak{g}$ arising from a transitive group action $\Lambda \colon G \times M \rightarrow M$ as
\[
	\lambda_*(u)(p) = \left.\frac{d}{dt}\right|_{t=0} \Lambda(\exp(tu),p).
\]

Although Runge--Kutta--Munthe-Kaas (RKMK) methods can be applied to non-autonomous problems, the framework for their analysis based on enveloping algebras of post-Lie algebras \cite{EFLMK15} has only been derived in the autonomous setting. It is a standard result for equations on $\mathbb{R}^n$ that non-autonomous problems may be autonomised by the addition of an extra variable corresponding to time. Whilst this results in an equivalent method (provided the RK coefficients in the Butcher tableau obey $c_i = \sum_j a_{ij}$, as is usually assumed), this can be used to simplify convergence and stability analysis.
Autonomisations of RKMK methods on $\mathbb{R}^n$ have been considered in \cite{Min07}, to our knowledge a similar procedure has yet to be carried out on more general homogeneous spaces.

We now perform such an autonomisation for the Lie group integration problem, for which the most convenient method is to augment the Lie group $G$ to $G\times\mathrm{Aff}(1)$. We recall the definition of the group of affine transformations
$$
	\mathrm{Aff}(1):=\bigg\{\begin{pmatrix}
				     x & t \\
				     0 & 1
			   \end{pmatrix} \ |\ x,t\in\mathbb{R}, \; x\neq 0 \bigg\}.
$$
The associated Lie algebra, $\mathfrak{aff}(1)$, is generated by the two matrices 
$$
	e_0=\begin{pmatrix}
				     0 & 1 \\
				     0 & 0
			   \end{pmatrix},\ 
e_{-1}=\begin{pmatrix}
				     1 & 0 \\
				     0 & 0
			   \end{pmatrix}.
$$
The first element corresponds to translation, indeed we have
$$
	\exp(t e_0) = \begin{pmatrix}
	1 & t \\
	0 & 1
	\end{pmatrix}.
$$
From this last result, the following is immediate regarding the non-autonomous linear initial value problem on $GL(n,\mathbb{R})$
\begin{equation}
\label{matrix:ode1}
	\dot{Y}(t) = A(t) Y(t),\quad Y(0)=y_0, \quad A(t)\in\mathfrak{gl}(n,\mathbb{R}).
\end{equation}
 
\begin{lemma}\label{lem:autonomization}
Problem \eqref{matrix:ode1} is equivalent to the following Lie group integration representation on $G$
\[
	\dot{\mathcal{Y}} = \mathcal{A}(\mathcal{Y})\cdot \mathcal Y,
	\quad \mathcal{Y} (0)=\begin{pmatrix}
				     y_0 & 0 \\
					0 & I
			   \end{pmatrix},
\]
where $G:=GL(n) \times \mathrm{Aff}(1)$, $\mathfrak{g}:=\mathfrak{gl}(n)\times\mathfrak{aff}(1)$, and $\mathcal Y(t) \in G$, $\mathcal{A} \colon G\rightarrow\mathfrak{g}$ take the form
\begin{equation}\label{Abar}
	\quad \mathcal Y=\begin{pmatrix}
				Y & 0 & 0 \\
				0 & x & t \\
				0 & 0 & 1
		      \end{pmatrix},\quad
	\mathcal{A}(\mathcal Y)=\begin{pmatrix}
				A(t) & 0 & 0 \\
				    0 & 0 & 1 \\
				    0 & 0 & 0
			     \end{pmatrix}.
\end{equation}
\end{lemma}

Note that this corresponds to the special case where the manifold $M=G$ and the group action on $G$ is simply defined in terms of the product in $G$. By abuse of notation, we identify the map $\mathcal{A} \colon G\rightarrow\mathfrak{g}$ with the vector field $\lambda_*(\mathcal{A})\in\mathcal{X}(G)$, the infinitesimal action of $\mathfrak{g}$ on $G$ being that associated to the right-multiplication action of $G$ on itself. In the autonomous case, this corresponds to the identification of $A\in\mathfrak{gl}(n)$ with the right-invariant vector field with tangent vector $A$ at the origin. In contrast, the vector fields $\mathcal{A}$ are only quasi-right-invariant in the following sense:

\begin{mydef}[Quasi-right-invariant fields]
\label{def:quasiright}
	We define $\mathfrak{t} \subset\mathcal{X}(G)$ to be the space of vector fields invariant by right-translations of the Lie subgroup $GL(n)\subset GL(n)\times\mathrm{Aff}(1)$ obtained by restricting to the identity element in $\mathrm{Aff}(1)$.
\end{mydef}

To illustrate this definition, suppose that $e_{-1},e_0,e_1,\ldots,e_d$ is a basis for the space of right-invariant vector fields $\mathfrak{g}$, where $e_{-1},e_0$ are the products of the basis vectors for $\mathfrak{aff}(1)$ with the zero vectors in $\mathfrak{gl}(n)$. The elements of $\mathfrak{t}$ comprise vector fields of the form
\[
	\mathcal V = v_0 e_0 + \sum_{i=1}^d v_i(t) e_i,
\]
or alternatively those that can be expressed through the infinitesimal action $\lambda_*$ as
\[
\tilde{\mathcal{A}}(A(t),a) = 
	\lambda_*\begin{pmatrix}
	A(t) & 0 & 0 \\
	0 & 0 & a \\
	0 & 0 & 0
	\end{pmatrix}
\]
for some constant $v_0=a$ and matrix $A(t) \in \mathfrak{gl}(n)$. For the sake of transparency we will omit the arguments of vector fields $\tilde{\mathcal{A}}$ whenever we can do so without causing confusion.

\begin{lemma}
\label{lem:Liesubalgebra}
The space $\mathfrak{t}$ of quasi-right-invariant vector fields forms a Lie subalgebra in $\mathcal{X}(G)$
\end{lemma}
\begin{proof}
To prove this we must compute the Jacobi bracket of vector fields
\[
\tilde{\mathcal{H}} = \lambda_*\begin{pmatrix}
H(t) & 0 & 0 \\
0 & 0 & h \\
0 & 0 & 0
\end{pmatrix} = h e_0 + \sum_i h_i(t) e_i
\]
\[
\tilde{\mathcal{K}} = \lambda_*\begin{pmatrix}
K(t) & 0 & 0 \\
0 & 0 & k \\
0 & 0 & 0
\end{pmatrix} = k e_0 + \sum_i k_i(t) e_i
\]
to which purpose we use the identification of vector fields with derivations on $C^{\infty}(G)$ functions. Indeed, employing summation convention on positive indices and using $x_i$ for derivatives of $\varphi$ in the direction of $e_i$ and $t$ for the direction of $e_0$, we have
\begin{align*}
	\llbracket \tilde{\mathcal H},\tilde{\mathcal K}\rrbracket (\varphi)
	&= \tilde{\mathcal H}(\tilde{\mathcal K}(\varphi)) - \tilde{\mathcal K}(\tilde{\mathcal H}(\varphi)) \\
	&= h \frac{\partial k_j}{\partial t}\frac{\partial\varphi}{\partial x_j} 
	- k \frac{\partial h_j}{\partial t}\frac{\partial\varphi}{\partial x_j}
	+ h_i \frac{\partial k_j}{\partial x_i} \frac{\partial\varphi}{\partial x_j}
	- k_i \frac{\partial h_j}{\partial x_i} \frac{\partial\varphi}{\partial x_j}.
\end{align*}
In coordinate free notation this gives the following, from which the proof of the claim is immediate:
\begin{equation}\label{tjacobi}
	\llbracket \tilde{\mathcal H},\tilde{\mathcal K}\rrbracket  = \lambda_*\begin{pmatrix}
	[H(t),K(t)] + h\dot{K}(t) - k\dot{H}(t) & 0 & 0 \\
	0 & 0 & 0 \\
	0 & 0 & 0
	\end{pmatrix}.
\end{equation}
\end{proof}
Following \cite{Postnikov}, we recall that the right Cartan connection $\tr$ on $\mathcal{X}(G)$ is defined by setting  $\tilde{\mathcal H} \tr \tilde{\mathcal K} := 0$ for right-invariant $\tilde{\mathcal H}$ and $\tilde{\mathcal K}$. Indeed, this extends uniquely to a connection $\tr$ on $\mathcal{X}(G)$ by writing a general vector field in the from $f_i \tilde{\mathcal H_i}$ for some $f_i \in C^{\infty}(G)$ and $\tilde{\mathcal H_i}$ a basis for $\mathfrak{g}$ and enforcing the equality
\begin{equation}
\label{connection}
	(f_i \tilde{\mathcal H_i}) \tr (g_j \tilde{\mathcal K_j}) 
	= f_i\,dg_j(\tilde{\mathcal H_i}) \tilde{\mathcal K_j} 
		+ f_i g_j(\tilde{\mathcal H_i} \tr \tilde{\mathcal K_j}).
\end{equation} 
This restricts to a bilinear product on $\mathfrak{t}$, indeed we compute
\begin{equation}\label{tconn}
	\tilde{\mathcal H}\triangleright \tilde{\mathcal K} = \lambda_* 
				\begin{pmatrix}
				h\frac{\partial K}{\partial t}(t) & 0 & 0\\
				    0 & 0 & 0 \\
				    0 & 0 & 0
			     	\end{pmatrix}.
\end{equation}
We are now in a position to define the Lie bracket $[\tilde{\mathcal{H}},\tilde{\mathcal{K}}]_{\mathfrak{t}} = \llbracket \tilde{\mathcal{H}},\tilde{\mathcal{K}} \rrbracket - \tilde{\mathcal{H}} \tr \tilde{\mathcal{K}} + \tilde{\mathcal{K}} \tr \tilde{\mathcal{H}}$ as per Example \ref{exm:connection}. From \eqref{tjacobi} and \eqref{tconn} we then find
\[
[\tilde{\mathcal{H}},\tilde{\mathcal{K}}]_{\mathfrak{t}} = 
\lambda_*\begin{pmatrix}
[H(t),K(t)]  & 0 & 0 \\
0 & 0 & 0 \\
0 & 0 & 0
\end{pmatrix}.
\]
Note that the scalar components $h,k$ of $\tilde{\mathcal H},\tilde{\mathcal K}$ have no impact on $[\tilde{\mathcal{H}},\tilde{\mathcal{K}}]_{\mathfrak{t}}$. 
\begin{lemma}\label{tPostLie}
The space $(\mathfrak{t},[\cdot,\cdot]_{\mathfrak{t}},\tr)$ is a post-Lie algebra, where $\tr$ is the right Cartan connection and $[\cdot,\cdot]_{\mathfrak{t}}$ is the negative of the associated torsion tensor. The two exponentials in the corresponding enveloping algebras are interpreted as follows: for any $\tilde{\mathcal{H}}$, $\exp^\centerdot(\tilde{\mathcal{H}} t).\tilde{\mathcal{K}}_0$ is the geodesic through the point $\tilde{\mathcal{K}}_0 \in G$ in the direction $\tilde{\mathcal{H}}(\tilde{\mathcal{K}}_0)$ and hence coincides with the matrix exponential. In contrast, $\exp^*(\tilde{\mathcal{H}} t).\tilde{\mathcal{K}}_0$ is the integral curve of $\tilde{\mathcal{H}}$ passing through $\tilde{\mathcal{K}}_0$, and hence solves the equation \eqref{matrix:ode}.
\end{lemma}

\begin{proof}
This is essentially a specialisation of Example \ref{exm:connection}, as the Cartan connection $\tr$ defined above is flat and has constant torsion (see \cite{Postnikov}). In this instance we can give the calculations explicitly. 
As per Example \ref{exm:connection}, flatness of the connection is expressed by
\[
	\tilde{\mathcal{H}} \tr (\tilde{\mathcal{K}} \tr \tilde{\mathcal{J}}) - \tilde{\mathcal{K}} \tr (\tilde{\mathcal{H}} \tr \tilde{\mathcal{J}}) 
	- \llbracket \tilde{\mathcal{H}},\tilde{\mathcal{K}} \rrbracket \tr \tilde{\mathcal{J}} = 0.
\]
To see this, we note that we need only check the top left block in the matrix representation of elements of $\mathfrak{t}$, which is
\[
	h \frac{\partial }{\partial t} (k \frac{\partial {J}}{\partial t})
	- k \frac{\partial }{\partial t} (h \frac{\partial {J}}{\partial t}) - 0 = 0.
\]
The remaining post-Lie axiom to check is the constant torsion relation
\[
	\tilde{\mathcal{H}} \tr [\tilde{\mathcal{K}},\tilde{\mathcal{J}}] 
	- [\tilde{\mathcal{K}}, \tilde{\mathcal{H}} \tr \tilde{\mathcal{J}}] 
	- [\tilde{\mathcal{H}}\tr \tilde{\mathcal{K}}, \tilde{\mathcal{J}}] = 0,
\]
which as all but the top left block of the associated matrix are zero reduces to the computation
\[
	h \left( \frac{\partial}{\partial t}[K, J] 
	- [K, \frac{\partial J}{\partial t}] 
	- [\frac{\partial K}{\partial t}, J]\right) = 0,
\]
as can be seen by expanding the commutators using Leibniz rule.
\end{proof}

We are now in a position to relate the classical and post-Lie Magnus expansions.

\begin{theorem}[Geometric Magnus expansion]
Let $\mathcal{A}\in\mathfrak{t}$ be associated to the $GL(n;\mathbb{R})$-valued function $A(t)$ as per (\ref{Abar}), i.e.
\begin{align*}
\mathcal{A}(t) &= \begin{pmatrix}
A(t) & 0 & 0 \\
0 & 0 & 1 \\
0 & 0 & 0
\end{pmatrix}.
\end{align*}
The solution of the equation defining the post-Lie Magnus expansion in the post-Lie algebra $(\mathfrak{t},[\cdot ,\cdot ],\tr)$,
\[
	\dot{\theta}({\mathcal{A}})(\tau) = \mathrm{dexp}^{\centerdot -1}_{\theta({\mathcal{A}})} 
	\big( \exp^\centerdot \big(\theta({\mathcal{A}})\big)\tr {\mathcal{A}} \big)(\tau) ,
\]
evaluated at initial time $t=0$, is given by 
$$
	\dot{\theta}({\mathcal{A}})(\tau)(0)= \begin{pmatrix}
							\Omega(A)(\tau) & 0 & 0\\
							    0 & 0 & 1 \\
							    0 & 0 & 0
			     				\end{pmatrix},
$$ 
where $\Omega(A)$ is the classical Magnus expansion.
\end{theorem}

\begin{proof}
We begin by noting that the identity \eqref{exprelation1} allows us to write $\exp^\centerdot(\theta(\mathcal{A})(\tau))\tr \mathcal{A} = \exp^\ast(\tau\mathcal{A})\tr \mathcal{A}$. Using \eqref{product} we find that 
$$
	\exp^\ast(\tau\mathcal{A})\tr \mathcal{A} = 
			  \mathcal{A}
			+ \tau\mathcal{A} \tr \mathcal{A} 
			+\tau^2 \frac{1}{2}\mathcal{A} \tr (\mathcal{A} \tr \mathcal{A}) 
			+ \tau^3\frac{1}{6}\mathcal{A} \tr (\mathcal{A} \tr (\mathcal{A} \tr \mathcal{A})) + \cdots.
$$
By the definition of the connection $\tr$, we have for all $n\geq 1$
\begin{equation}
\label{beauty}
	\ell^{(n)}_{\tau \mathcal{A}\tr}(\mathcal{A})(t) = \tau
		\mathcal{A} \tr ( \ell^{(n-1)}_{\tau \mathcal{A}\tr}(\mathcal{A}))(t)=
 				\begin{pmatrix}
					{\tau}^{n}\frac{\partial^{n} A}{\partial t^{n}}(t) & 0 & 0\\
					    0 & 0 & 0 \\
					    0 & 0 & 0
			    	 \end{pmatrix},
\end{equation} 
from which it follows that
\[
	(\exp^\ast(\tau{\mathcal{A}})\tr {\mathcal{A}})(t) = \begin{pmatrix}
							\sum_{k \ge 0} \frac{\tau^k}{k!}A^{(k)}(t) & 0 & 0\\
							    0 & 0 & 1 \\
							    0 & 0 & 0
			     				\end{pmatrix}.
\]
In particular, $\exp^\ast(\tau{\mathcal{A}})\tr {\mathcal{A}}$ is the vector field which evaluates to ${\mathcal{A}}(t+\tau)$ at time $t$. Now $\theta({\mathcal{A}})(\tau)$ is a vector field in $\mathfrak{t}$ (see Lemma \ref{lem:Liesubalgebra}), which we evaluate at time $t=0$ (the position in space is unimportant due to (quasi-)right-invariance defined in Definition \ref{def:quasiright}). By the above, the differential equation may then be re-written
\[
	\dot{\theta}({\mathcal{A}})(\tau)(0) 
	= \mathrm{dexp}^{\centerdot -1}_{\theta(\mathcal{A})}( \mathcal{A})(\tau)(0).
\]
By expanding the series for $\mathrm{dexp}^{\centerdot -1}$ in the Lie bracket on $\mathfrak{t}$, it follows that the component of $\theta({\mathcal{A}})(\tau)$ in $\mathfrak{aff}(1)$ is simply $e_0$, and hence the equation for the remaining component in $GL(n;\mathbb{R})$ reduces to the classical Magnus equation (\ref{cme:omega}).
\end{proof}


\section{Runge--Kutta versus Magnus integration methods}
\label{sect:RKMKMag}

The post-Lie Magnus expansion is a key ingredient in Runge--Kutta--Munthe-Kaas (RKMK) schemes for the numerical solution of differential equations on homogeneous spaces \cite{EFLMK15}. In this section we illustrate by means of examples that Magnus integrators may be seen as special instances of RKMK methods applied to equations of the type
\[
	\dot{Y} = A(t) Y,\quad Y(0)=Y_0,
\]
where $A,Y \in\mathbb{R}^{m\times m}$. We recall that when applying an RKMK method to such an equation, the internal stages (with $i=1,\ldots,s$) take the form
\begin{eqnarray}
		 u_i &=& h \sum_{j=1}^s a_{ij} \tilde{f}_j \label{ustages} \\
	\tilde{f}_i &=& \mathrm{d}\!\exp^{-1}_{u_i} A(t_0 + hc_i) \nonumber \\
		      &=& A(t_0 + hc_i) - \frac{1}{2}[u_i, A(t_0 + hc_i)] + \cdots, \label{fstages}
\end{eqnarray}
and the scheme progresses by
\begin{equation} \label{update}
	Y_1 = \exp\Big(h\sum_{i=1}^s b_i \tilde{f}_i\Big)Y_0,
\end{equation}
where $a_{ij},b_i,c_i$ form the Butcher tableau of the underlying RK scheme. We evaluate the consequences of two two-stage methods: the Heun scheme and the two stage Gauss--Legendre method. The corresponding tableaux are:
\[
	\begin{array}{r|rr}
	0 & 0 & 0 \\
	1 & 1 & 0 \\
	\hline\\[-0.3cm]
	& \frac{1}{2} & \frac{1}{2}
	\end{array}
	\qquad
	\begin{array}{r|rr}
	\frac{1}{2}-\omega & \frac{1}{4} & \frac{1}{4}-\omega \\[0.1cm]
	\frac{1}{2}+\omega & \frac{1}{4}+\omega & \frac{1}{4} \\[0.1cm]
	\hline\\[-0.3cm]
	& \frac{1}{2} & \frac{1}{2}
	\end{array}
\]
where Heun is on the left and Gauss--Legendre on the right, with $\omega=\frac{\sqrt{3}}{6}$.\\
\newline
\smallskip\textbf{Heun scheme}\\
This method is explicit. Indeed we will begin by noting that $u_1=0$, so that
\[
	\tilde{f}_1 = \mathrm{d}\!\exp^{-1}_{0}(A)(t_0) = A(t_0).
\]
We then find that $u_2 = h A(t_0)$, and hence
\[
	\tilde{f}_2 = \mathrm{d}\!\exp^{-1}_{h A(t_0)}(A)(t_1) \approx A(t_1) - \frac{h}{2}[A(t_0),A(t_1)].
\]
Using the two term approximation for $\mathrm{d}\!\exp^{-1}$ above, we therefore obtain the integration scheme
\[
	Y_1 = \exp\Big(\frac{h}{2}(A_0 + A_1) - \frac{h^2}{4}[A_0,A_1]\Big)Y_0,
\]
where $A_0 := A(t_0)$ and $A_1 := A(t_1)$. It can be shown that this is the Magnus integrator obtained from taking the first two terms in the Magnus series and approximating the integrals by the trapezoidal rule.\\
\newline
\smallskip\textbf{Gauss--Legendre scheme}\\
Here the situation is somewhat more complicated, as the method is implicit. Indeed, for simplicity write $\tau_i = t + hc_i$, $i=1,2$. We will then use the shorthand $A_i := A(\tau_i)$, giving the system
\begin{eqnarray*}
		u_1 &=& \frac{h}{4}\tilde{f}_1 + h(\frac{1}{4}-\omega)\tilde{f}_2 \\
		u_2 &=& h(\frac{1}{4}+\omega)\tilde{f}_1 + \frac{h}{4}\tilde{f}_2 \\
	\tilde{f}_1 &=& \mathrm{d}\!\exp^{-1}_{u_1} A_1 \\
	\tilde{f}_2 &=& \mathrm{d}\!\exp^{-1}_{u_2} A_2.
\end{eqnarray*}
If we truncate the series for $\mathrm{d}\!\exp^{-1}$ to two terms once again, we find
\begin{eqnarray*}
	\tilde{f}_1 &=& A_1 - \big[\frac{h}{4}\tilde{f}_1 + h(\frac{1}{4}-\omega)\tilde{f}_2, A_1\big] \\
	\tilde{f}_2 &=& A_2 - \big[h(\frac{1}{4}+\omega)\tilde{f}_1 + \frac{h}{4}\tilde{f}_2,A_2\big].
\end{eqnarray*}
Suppose we solve the above system by a single function iteration starting from the first order approximation $\tilde{f}_i = A_i$. This gives the explicit forms
\begin{eqnarray*}
	\tilde{f}_1 &=& A_1 + \frac{h}{2}(\frac{1}{4}-\omega)[A_1, A_2] \\
	\tilde{f}_2 &=& A_2 - \frac{h}{2}(\frac{1}{4}+\omega)[A_1,A_2].
\end{eqnarray*}
The resulting integration scheme is (inserting the value $\omega=\frac{\sqrt{3}}{6}$)
\[
	Y_1 = \exp\Big(\frac{h}{2}(A_1 + A_2) - \frac{h^2\sqrt{3}}{12}[A_1,A_2]\Big)Y_0,
\]
which is identical to the standard order four Magnus integrator obtained from the two point Gauss quadrature \cite{BCOR2009}.


\bigskip\noindent\textbf{Continuous stage methods}\\
In his monograph from 1987, Butcher \cite[Section 385]{Butcher87} suggested to study the exact solution of an ordinary differential equation obtained by applying a continuous stage Runge--Kutta method. Sums are replaced by integrals, this means that rather than indexing the stages $u_i, \tilde{f}_i$ in (\ref{ustages}--\ref{fstages}), one may consider a continuum of stages $u(\tau), \tilde{f}(\tau)$, $\tau \in[0,1]$ and introduce continuous Runge--Kutta coefficients $a(\tau,\sigma), b(\tau)$. 
By choosing 
\begin{equation} 
\label{excontstage}
	b(\tau)\equiv 1,\quad
	a(\tau,\sigma)=\left\{\begin{array}{ll}1 & 0\leq\sigma\leq \tau \\ 0 & \text{otherwise} \end{array}\right.\
\end{equation} 
the resulting Runge--Kutta scheme will produce the exact solution to the ordinary differential equation. 

More recently, there has been a growing interest in studying such continuous stage schemes in their own right, for instance, the averaged vector field method (see, e.g., \cite{COS2014}) can be written in such a format, and more general studies of continuous stage Runge--Kutta methods have been undertaken by several authors (see, e.g., \cite{TS2014}). By inserting the coefficients of a continuous stage method into (\ref{ustages}--\ref{update}) and eliminating $\tilde{f}(\tau)$ we obtain
\begin{align*}
	u(\tau) &= h\int_0^1 a(\tau,\sigma) \mathrm{d}\!\exp^{-1}_{u(\sigma)}  A(t_0+hc(\sigma))\,\mathrm{d}\sigma \\
		v &= h\int_0^1 b(\tau)  \mathrm{d}\!\exp^{-1}_{u(\tau)}  A(t_0+hc(\tau))\,\mathrm{d}\tau \\
	    Y_1 &= \exp(v)\cdot Y_0,
\end{align*}
where $c(\tau)=\int_0^1 a(\tau,\sigma)\,\mathrm{d}\sigma$. It is perhaps no surprise that if $a(\tau,\sigma)$ and $b(\tau)$ are chosen as in \eqref{excontstage} one finds that $u(\tau)\equiv \Omega(A)(t_0+\tau h)$ in \eqref{cme:omega} as well as $v=\Omega(A)(t_0+h)$ so the exact Magnus series expansion is recovered. One may also note that deformations of the Magnus expansion can be obtained by choosing other coefficient functions.



\begin{thebibliography}{99}
\parskip1.0ex

\bibitem{AG1981}
{\sc A.~Agrachev and R.~Gamkrelidze}, 
{\em{Chronological algebras and nonstationary vector fields}}, 
 J.~Sov.~Math.~{\bf{17}} (1981), 1650--1675.
 
\bibitem{BandSch}
{\sc R.~Bandiera and F.~Sch\"atz}
{\emph{Eulerian idempotent, pre-Lie logarithm and combinatorics of trees}},
arXiv:1702.08907. 
 
\bibitem{BCOR2009}
{\sc S.~Blanes, F.~Casas, J.A.~Oteo and J.~Ros},
{\em{Magnus expansion: mathematical study and physical applications}},
Phys.~Rep.~{\bf{470}} (2009), 151--238.

\bibitem{Butcher87}
{\sc J.~C.~Butcher},
{\em The Numerical Analysis of Ordinary Differential Equations: 
Runge--Kutta and General Linear Methods}, John Wiley \& Sons, (1987).

\bibitem{Cartier2011}
{\sc P.~Cartier},
{\em{Vinberg algebras, Lie groups and combinatorics}},
Clay Mathematical Proceedings {\bf{11}} (2011), 107--126. 

\bibitem{CHO2014}
{\sc E.~Celledoni, H.~Marthinsen and B.~Owren},
{\em{An introduction to Lie group integrators - basics, new developments and applications}}, 
Journal of Computational Physics {\bf{257}}, Part B (2014), 1040--1061.
	
\bibitem{COS2014}
{\sc E.~Celledoni, B.~Owren and Y.~Sun},
{\em{The minimal stage, energy preserving {R}unge-{K}utta method for 
polynomial {H}amiltonian systems is the averaged vector field method}}, 
Math.~Comp.~{\bf{83}} (2014), 1689--1700.	
	
\bibitem{Chapoton02}
{\sc F.~Chapoton},
{\em{Rooted trees and an exponential-like series}},
arXiv:math/0209104

\bibitem{Chapoton09}
{\sc F.~Chapoton}, 
{\em{A rooted-trees q-series lifting a one-parameter family of Lie idempotents}},
Algebra \& Number Theory {\bf{3}} (2009), 611--636.

\bibitem{CP13}
{\sc F.~Chapoton and F.~Patras}, 
{\em{Enveloping algebras of preLie algebras, Solomon idempotents and the Magnus formula}}, 
Int.~J.~Algebra and Computation {\bf{23}}, No.~4 (2013), 853--861.   	

\bibitem{CHV10}
{\sc Ph.~Chartier, E.~Hairer and G.~Vilmart}, 
{\em{Algebraic structures of B-series}} 
Found.~Comput.~Math.~{\bf{10}} (2010), 407--427. 	
	
\bibitem{Voss1}
{\sc C.~Curry, K.~Ebrahimi-Fard and H.~Munthe-Kaas},
{\em{What is a post-Lie algebra and why is it useful in geometric integration}},
``Numerical Mathematics and Advanced Applications -- ENUMATH 2017", Springer's Lecture Notes in Computational Science and Engineering {\bf{126}}, 2018. arXiv:1712.09415

\bibitem{EFLMK15}
{\sc K.~Ebrahimi-Fard, A.~Lundervold and H.~Z.~Munthe-Kaas},
{\em{On the Lie enveloping algebra of a post-Lie algebra}}, 
J.~Lie Theory {\bf{25}} No.~4 (2015), 1139--1165.

\bibitem{EFLMMK15}
{\sc K.~Ebrahimi-Fard, A.~Lundervold, I.~Mencattini and H.~Z.~Munthe-Kaas},
{\em{Post-Lie Algebras and Isospectral Flows}}, 
SIGMA {\bf{25}} No.~11 (2015), 093.	

\bibitem{EFMMK17}
{\sc K.~Ebrahimi-Fard, I.~Mencattini and H.~Z.~Munthe-Kaas},
{\em{Post-Lie algebras and factorization theorems}},
J.~Geom.~Phys.~{\bf{119}} (2017), 19--33.

\bibitem{EFM2018}
{\sc K.~Ebrahimi-Fard and I.~Mencattini},
{\em{Post-Lie Algebras, Factorization Theorems and Isospectral Flows}}.
In: Ebrahimi-Fard K., Barbero Li\~n\'an M.~(eds) Discrete Mechanics, Geometric Integration and Lie--Butcher Series. Springer Proceedings in Mathematics \& Statistics, {\bf{267}}. Springer, Cham, 2018.

\bibitem{EM2009}
{\sc K.~Ebrahimi-Fard and D.~Manchon},
{\em{A Magnus- and Fer-type formula in dendriform algebras\/}}, 
Found.~Comput.~Math.~{\bf{9}} (2009), 295--316. 

\bibitem{Epstein_63}
{\sc I.J.~Epstein},
{\em{Conditions for a Matrix to Commute with its Integral}},
Proc.~American Math.~Soc.~{\bf{14}} (1963), 266--270.

\bibitem{AlK_14}
{\sc M.~J.~Hasan Al-Kaabi},
{\em{Monomial bases for free pre-Lie algebras}},
S\'eminaire Lotharingien de Combinatoire {\bf{71}} (2014), Article B71b.

\bibitem{AlKManPat_17}
{\sc M.~J.~Hasan Al-Kaabi, D.~Manchon, F.~Patras},
{\em{Monomial bases and pre-Lie structure for free Lie algebras}},
arXiv:1708.08312.

\bibitem{Iserles00}
{\sc A.~Iserles, H.~Z.~Munthe-Kaas, S.~P.~N{\o}rsett and A.~Zanna},
{\em{Lie-group methods}},
Acta Numer.~{\bf{9}} (2000), 215--365.

\bibitem{Iserles02}
{\sc A.~Iserles},
{\em{Expansions that grow on trees}},
Notices of the AMS {\bf{49}} (2002), 430--440.

\bibitem{LMK13}
{\sc A.~Lundervold and H.~Z.~Munthe-Kaas},
{\em{On post-Lie algebras, Lie--Butcher series and moving frames}},
Found.~Comput.~Math.~{\bf{13}} Issue 4 (2013), 583--613.

\bibitem{Magnus1954}
{\sc W.~Magnus},
{\em{On the exponential solution of differential equations for a linear operator}},
Commun.~Pure Appl.~Math.~{\bf{7}} (1954), 649--673.

\bibitem{Manchon2011}
{\sc D.~Manchon},
{\em{A short survey on pre-Lie algebras}}, 
in Noncommutative Geometry and Physics: Renormalisation, Motives, Index Theory, ({\sc A.~Carey} ed.), E.~Schr\"odinger Institut Lectures in Math.~Phys., Eur.~Math.~Soc.~(2011).

\bibitem{Meinrenken2013}
{\sc E.~Meinrenken},
Clifford Algebras and Lie Theory,
Ergebnisse der Mathematik und ihrer Grenzgebiete, vol.~58. Springer 2013.

\bibitem{MielPleb70}
{\sc B.~Mielnik and J.~Pleba\'nski},
{\em{Combinatorial approach to Baker--Campbell--Hausdorff exponents}}
Ann.~Inst.~Henri Poincar\'e A {\bf{XII}} (1970), 215--254.

\bibitem{Min07}
{\sc B.~Minchev},
{\em {Lie group integrators with nonautonomous frozen vector fields}},
Int. J. Comput. Sci. Eng. \textbf{3}(4) (2007). 

\bibitem{MuOw99}
{\sc H.~Munthe-Kaas and B.~Owren},
{\em{Computations in a free Lie algebra}},
Phil.~Trans.~R.~Soc.~Lond.~A {\bf{357}} (1999), 957--981.

\bibitem{Murua06}
{\sc A.~Murua}, 
{\em{The Hopf Algebra of Rooted Trees, Free Lie Algebras, and Lie Series}},
Found.~Comput.~Math.~{\bf{6}} issue 4 (2006), 387--426. 

\bibitem{OG2008}		
{\sc J.-M.~Oudom and D.~Guin},
 {\em{On the Lie enveloping algebra of a pre-Lie algebra}},
Journal of K-theory: K-theory and its Applications to Algebra, Geometry, and Topology {\bf{2}}, number 1 (2008), 147--167.

\bibitem{Postnikov}
{\sc M.~M.~Postnikov}, 
{\em Geometry VI: Riemannian Geometry},
Enc.~Math.~Sci.~{\bf{91}}, Springer, Berlin (2001)

\bibitem{TS2014}
{\sc W.~Tang and Y.~Sun},
{\em Construction of Runge--Kutta type methods for solving ordinary differential equations},
Applied Mathematics and Computation {\bf{234}} (2014), 179--191.

\bibitem{Wilcox67}
{\sc R.~M.~Wilcox},
{\em{Exponential operators and parameter differentiation in quantum physics}},
J.~Math.~Phys.~{\bf{8}} (1967), 962--982.

\end{thebibliography}
\end{document}